\documentclass[12pt, a4paper]{amsart}
\usepackage[colorlinks,backref]{hyperref}
\usepackage{enumerate, longtable}
\usepackage{amsmath, amscd, amsfonts, amsthm, amssymb, latexsym, comment, stmaryrd, graphicx}
\usepackage{xcolor}
\usepackage[all]{xy}
\usepackage{a4wide}
\usepackage{tikz}
\usetikzlibrary{decorations.markings,calc}

\theoremstyle{plain}
\newtheorem{theorem}{Theorem}[section]
\newtheorem{proposition}[theorem]{Proposition}
\newtheorem{lemma}[theorem]{Lemma}

\theoremstyle{definition}
\newtheorem{definition}[theorem]{Definition}

\newcommand{\EM}{{\mathcal{M}}}

\newcommand{\ZZ}{\mathbb{Z}}

\newcommand{\FF}{\mathbb{F}}

\newcommand{\B}{\mathcal{B}}
\newcommand{\M}{\mathcal{M}}
\newcommand{\F}{\mathcal{F}}
\newcommand{\Legendre}[2]{\left(\frac{#1}{#2}\right)}

\renewcommand{\Re}{\mathrm{Re}}
\renewcommand{\Im}{\mathrm{Im}}
\begin{document}

\title{On the Function Field Analogue of Landau's Theorem on Sums of Squares}
\author{Lior Bary-Soroker}
\address{Raymond and Beverly Sackler School of Mathematical Sciences, Tel Aviv University, Tel Aviv 69978, Israel}
\email{barylior@post.tau.ac.il}

\author{Yotam Smilansky}
\address{Raymond and Beverly Sackler School of Mathematical Sciences, Tel Aviv University, Tel Aviv 69978, Israel}
\email{yotamsmi@post.tau.ac.il}

\author{Adva Wolf}
\address{Raymond and Beverly Sackler School of Mathematical Sciences, Tel Aviv University, Tel Aviv 69978, Israel}
\email{wolf.adva@gmail.com}

\subjclass[2010]{}

\maketitle

\begin{abstract}
This paper deals with function field analogues of the famous theorem of Landau which gives the asymptotic density of sums of two squares in $\ZZ$. 

We define the analogue of a sum of two squares in $\FF_q[T]$, $q$ odd and 
estimate the number $B_q(n)$ of such polynomials of degree $n$ in two cases. The first case is when $q$ is large and $n$ fixed and the second case is  when $n$ is large and $q$ is fixed. Although the methods used and main terms computed in each of the two cases differ, the two iterated limits of (a normalization of) $B_q(n)$ turn out to be exactly the same.
\end{abstract}

\section{Introduction}
\subsection{Landau's Classical Theorem Regarding Sums of Two Squares}
Let $b(n)$ be the characteristic function of integers that are representable as a sum of two squares and let 
\[
B(x) =\sum_{n\leq x} b(n)
\]
be the number of such integers up to $x$. 
Landau's Theorem \cite{Landau} gives an asymptotic formula for $B(x)$:
\begin{equation}\label{thm:Landau}
B(x) =  K \frac{x}{\sqrt{\log x}} + O\left( \frac{x}{\log^{3/2} x}\right), \qquad x\to \infty,
\end{equation}
where 
\begin{equation}
K = \frac{1}{\sqrt{2}}\prod_{p \equiv 3 \hskip-7pt\pmod 4}(1-p^{-2})^{-1/2}\approx 0.764  
\end{equation}
is the Landau-Ramanujan constant.

In this work we study function field analogues of Landau's theorem in the two limits of a large degree and of a large finite field.

\subsection{Function Fields}
Let $q$ be a prime power. \emph{We always assume that $q$ is odd.} Denote by $\FF_q[T]$ the ring of polynomials over the finite field $\FF_q$ and by 
\[
\EM_{n,q} = \{f\in \FF_q[T] : f=T^n + a_1 T^{n-1} + \cdots + a_n\}
\] 
the subset of monic polynomials of degree $n$.  

To define the function field analogue of a sum of two squares, we recall that an integer is a sum of two squares if and only if it is a norm of a Gaussian integer, i.e., $n=a^2+b^2$ if and only if $n={\rm Norm}(a+bi)$,  $a+bi\in \ZZ[i]$. Thus we define:

\begin{definition}
Let $q$ be an odd prime power. For a polynomial $f\in \M_{n,q}$ we define the characteristic function:
\begin{equation}\label{eq:defbq}
b_q(f) = 
\begin{cases}
1, & f=A^2+TB^2 \mbox{ for } A,B\in \FF_q[T],\\
0, &\mbox{otherwise.}
\end{cases}
\end{equation}
and the counting function: 
\begin{equation}\label{eq:defBq}
B_q(n) = \sum_{f\in \EM_{n,q}} b_q(f).
\end{equation}
\end{definition}
In other words, $b_q(f)=1$ if and only if $f$ is a norm of an element of the ring extension $\FF_q[\sqrt{-T}]$, that plays the role of $\ZZ[i]$ in the function field  setting, and $B_q(n)$ counts the number of monic norms of degree $n$. 

We remark that it also makes perfect sense to consider the more general rings $\FF_q[\sqrt{\alpha T}]$, for a nonzero $\alpha\in \FF_q$, as analogues of $\ZZ[i]$. Then a norm will have the form $A^2-\alpha TB^2$. The theory and results in this generalization are exactly the same; so for the sake of simple exposition we restrict to $\alpha=-1$.

Another possibility for an analogue notion of sums of two squares is the naive one; namely, $A^2+B^2$. When $q\equiv 1\pmod 4$, then $-1= u^2$ for some $u\in \FF_q$, so 
\[
A^2+B^2=(A+uB)(A-uB)
\] 
and the problem becomes linear and thus easy. Regarding the case $q\equiv 3\pmod 4$, or more generally considering $A^2+\alpha B^2$ for $\alpha\in \FF_q$ not a square, related problems were considered in \cite{MW,HMW}. Our methods work in this case too, with changes that come from taking $\FF_{q^2}[T]$ as the analogue of the ring $\ZZ[i]$ instead of $\FF_q[\sqrt{-T}]$; however, we do not pursue this direction here.

It is interesting to find an asymptotic formula for $B_q(n)$ as $q^n\to \infty$. In this work we obtain such formula in each of the two sub-limits:     \textbf{the large finite field limit} $q\to \infty$ and \textbf{the large degree limit} $n\to \infty$. 

We remark that in characteristic $2$, $b_q(f)=1$ for all polynomials, since raising to a square is a homomorphism. 

\subsection{Large Finite Field Limit} 
In this limit we obtain the first two terms in the asymptotic formula for $B_q(n)$:
\begin{theorem}\label{thm:LF}
For every $n\geq 2$, 
\begin{equation}\label{LandauLF}
B_q(n) = 
\frac{1}{4^n} \binom{2n}{n} q^n + c_n q^{n-1} + O_n(q^{n-2}), \qquad q\to \infty,
\end{equation}
where 
\begin{equation}
c_n=\frac{1}{2\cdot4^{n-1}} \binom{2(n-1)}{n-1}  +\frac{1}{4^{n-1}} \binom{2(n-2)}{n-2}.
\end{equation}
\end{theorem}

Our proof is combinatoric/probabilistic in nature and is based on Ewens' sampling formula\footnote{
This formula was first used in 1972 in the study of the sampling distribution of allele frequencies in a population undergoing neutral selection (see \cite{Ewens} for the original paper). It has since been applied in a variety of fields, see the recent survey paper \cite{Crane}. 
%
} and on the Riemann hypothesis for rational function fields, which is an elementary theorem; that is, we use a prime number theorem for arithmetic progressions modulo $T$ with an explicit error term, see \eqref{PPTAP}.

\subsection{Large Degree Limit}
We say that a polynomial $P\in \FF_q[T]$ is \textbf{prime} if it is irreducible and monic. We denote the \textbf{Legendre symbol} (see \cite[Chapter 1]{Rosen}) by 
\begin{equation}\label{LegendreSymb}
\Legendre{f}{P} = f^{\frac{|P|-1}{2}} \pmod P,
\end{equation}
where $|P|=q^{\deg P}$ (see \cite[Chapter 3]{Rosen} for details). Note that $\Legendre{f}{P}\in \{0,1,-1\}$. 
Then we obtain:
\begin{theorem}\label{thm:xLandauLD}
Let $q$ be an odd prime power. Then 
\begin{equation}\label{eq:thm:LandauLD}
B_q(n) = \frac{K_q}{\sqrt{\pi}} \cdot\frac{q^n}{\sqrt{n}}  + O_q\Big(\frac{q^n}{n^{3/2}}\Big), \qquad n\to \infty
\end{equation}
where 
\begin{equation}\label{eq:Kq}
K_q = \left(1-q^{-1}\right)^{-\frac12} \prod_{\Legendre{P}{T}=-1}(1-|P|^{-2})^{-\frac12}. 
\end{equation}
\end{theorem}
We note that the asymptotic formulas for $B_q(n)$ in the different limits   differ from each other.

One can easily see the agreement of \eqref{thm:Landau} and \eqref{eq:thm:LandauLD} where $x$ is replaced by $q^n=\#\EM_{n,q}$, $\log x$ by $n$ and $K$ by $K_q/\sqrt{\pi}$. 
We emphasize that in the function field setting, we have an extra factor of $1/\sqrt{\pi}$. 
To explain this difference, we note that the main term in \eqref{eq:thm:LandauLD} arises from integrating $D(s)$ around $s=1$ (see \eqref{eq:G2Gbar2}) and similarly in the number field setting. Since $D(s)$ has the square root of the $L$-function of the quadratic character as a factor (see \eqref{eq:D2}) which is analytic in a neighbourhood of $s=1$, and similarly in the number field setting, the extra $1/\sqrt{\pi}$ factor can be explained by the different values of the corresponding $L$-functions at $s=1$. 

The agreement is even deeper since the function field proof goes tightly along the lines of the proof in the classical setting.

\subsection{Comparison Between Large Field and Large Degree Limits}
In order to compare the limits we will consider the iterated limits of a normalized counting function $B_q(n)$. 

We note that
\[
\left(1-q^{-1}\right)^{-\frac12} = 1+O\big(q^{-1}\big)
\]
and since $|P|= q^{\deg P}$, we have 
\[
\sum_{\Legendre{P}{T}=-1}\log(1-|P|^{-2})^{-1/2}\ll \sum_{\Legendre{P}{T}=-1}|P|^{-2} \ll \sum_{d=1}^\infty q^{-2d} \sum_{P, \deg P=d}1
\ll \sum_{d=1}^{\infty} q^{-d} = O(q^{-1}).
\]
Thus
\[
\begin{split}
K_q &= \left(1-q^{-1}\right)^{-\frac12} \prod_{\Legendre{P}{T}=-1}(1-|P|^{-2})^{-1/2} \\
&= \left(1-q^{-1}\right)^{-\frac12}\exp\left(\sum_{\Legendre{P}{T}=-1}\log(1-|P|^{-2})^{-1/2} \right)
\\
&= (1+O(q^{-1}))e^{O(q^{-1})}=1+O(q^{-1}).
\end{split}
\]
So, by \eqref{eq:thm:LandauLD} we have 
\begin{equation}\label{eq:limitqn}
\lim_{q\to \infty} \lim_{n\to \infty} \frac{B_q(n)}{q^n/\sqrt{n}}=\lim_{q\to \infty} \frac{K_q}{\sqrt \pi} = 1/\sqrt{\pi}.
\end{equation}

To calculate the other iterated limit, we recall that by either Wallis' product or Stirling's formula, one has  
\[
\binom{2n}{n}\sim \frac{4^n}{\sqrt{\pi n}}, \qquad n\to \infty.
\]
So, by \eqref{LandauLF} we have 
\begin{equation}\label{eq:limitnq}
\lim_{n\to\infty} \lim_{q\to \infty} \frac{B_q(n)}{q^n/\sqrt{n}} = \lim_{n\to \infty} \frac{\sqrt{n}}{4^n} \binom{2n}{n}  = 1/\sqrt{\pi}.
\end{equation}
From \eqref{eq:limitqn} and \eqref{eq:limitnq} we conclude that the asymptotic formulas for $B_q(n)$ in the two regimes agree.

We do not know of an asymptotic formula for $B_q(n)$ in any more general sub-limits of $q^n\to \infty$. 

\section*{Acknowledgments}
The authors wish to thank Alexei Entin, Avner Kiro, Ron Peled and Zeev Rudnick for helpful discussions and the anonymous referee for her/his valuable remarks. 

The research leading to these results was partially supported by the Israel Science Foundation (grant No.
925/14).

\section{Elementary Theory}
Fix an odd prime power $q$ and put $S=\sqrt{- T}$. This defines an embedding $\FF_q[T]\subseteq \FF_q[S]$ of the corresponding univariate polynomial rings, namely the representation of $f(T)\in \FF_q[T]$ in $\FF_q[S]$ is $f(T) = f(-S^2)$. 
There are two automorphisms of $\FF_q[S]$  fixing $\FF_q[T]$ that are induced from the maps $S\mapsto \pm S$. We call the nontrivial automorphism \textbf{conjugation}. The norm map 
\[
N\colon \FF_q[S]\to \FF_q[T]
\]
is then defined by 
\[
N(h(S))=h(S)h(-S)\in \FF_q[T], \qquad h\in \FF_q[S].
\]
Each element in $\FF_q[S]$ can be represented uniquely as $A+SB$ with $A,B\in \FF_q[T]$ and in this representation the norm map takes the form 
\[
N(A+SB)=(A+SB)(A-SB)=A^2 + T B^2.
\]
Recalling the definitions of $b_q(f)$ in \eqref{eq:defbq}, we immediately get that for $f\in \EM_{n,q}$, $b_q(f) = 1$ if and only if $f$ is a norm (i.e.\ $f= N(h)$ for some $h\in \FF_q[S]$). 
Since the norm map is multiplicative, in order to characterize representable $f$'s we need to understand the primes of $\FF_q[S]/\FF_q[T]$. This is a straightforward analogue of the classical setting in which one studies primes of $\ZZ[i]$ in order to understand sums of two squares. The key point in the classical theory is that a prime number $p$ is inert in $\ZZ[i]$ if and only if $p\equiv 3\pmod 4$. We now develop the function field counterpart, which is easier than the number field case, but we include the full details due to lack of reference.

Recalling that for a prime polynomial $P$, $\Legendre{\bullet}{P}$ is the Legendre symbol as defined in \eqref{LegendreSymb}, we have \cite[Proposition 3.2]{Rosen}
\[
\Legendre{-1}{P} = (-1)^{\frac{q-1}{2} \deg P}.
\] 
Thus, if $P\neq T$, by the quadratic reciprocity law \cite[Theorem~3.3]{Rosen},
\begin{equation}\label{LegendrePT}
\Legendre{P}{ T} \Legendre{- T}{P}= \Legendre{P}{T} \Legendre{T}{P} (-1)^{\frac{q-1}{2} \deg P} = (-1)^{\frac{q-1}{2}\deg P}(-1)^{\frac{q-1}{2} \deg P} = 1.
\end{equation}

\begin{lemma}\label{lem:inertinquadratcextension}
Let $P(T)\in \FF_q[T]$ be a prime polynomial of degree $n$. Then $P(T)=P(-S^2)$ is reducible in $\FF_q[S]$ if and only if $b_q(P)=1$.  Moreover if $P=A^2 + T B^2$ for some $A,B\in \FF_q[T]$, then  $A\pm S B$ are irreducible in $\FF_q[S]$. 
\end{lemma}

\begin{proof}
If $b_q(P)=1$, then we have 
\[
P=A^2+ T B^2 = (A+SB)(A-SB)
\]
with either $\deg_T A>0$ or $B$ nonzero. Note that $\deg_S(A)=2\deg_T(A)$ while $\deg_S(SB)=1+2\deg_T(B)$ is odd, hence $\deg_S(A\pm S B) =\max\{\deg_S(A),\deg_S(SB)\} > 0$, so $P$ is reducible in $\FF_q[S]$. 

On the other hand, if $P$ is reducible in $\FF_q[S]$, then 
\[
P=(A + BS)(C + DS).
\]
with $A,B,C,D\in \FF_q[T]$ such that both $A+BS$ and $C+DS$ are of positive $S$-degree. Applying the norm map gives
\[
P^2 = (A^2+ TB^2)(C^2+T D^2).
\]
From the unique factorization in $\FF_q[T]$, we have $ P = c(A^2+T B^2)$, for some  nonzero $c\in \FF_q$. Comparing leading coefficients we get that $c$ is a square, hence by replacing $A,B$ by $A/\sqrt{c},B/\sqrt{c}$, respectively, we get that $b_q(P)=1$,  as needed. 

Finally, assume $P=(A+SB)(A-SB)$. If $C+SD$ divides $A+SB$, then by taking norms, $C^2+ T D^2$ divides $A^2+ T B^2 = P$ in $\FF_q[T]$. Since $P$ is irreducible in $\FF_q[T]$, either $C$ is constant and $D=0$ or $C^2+ T D^2=cP$ for some $c\in \FF_q$, which implies that $\deg_S (C+SD)=\deg_S(A+SD)$, so $C+SD = c'(A+SD)$, for some $c'\in \FF_q$ and $A+SB$ is irreducible, as needed. 
\end{proof}

\begin{lemma}\label{lem:irreducibility}
If $T\neq P\in \FF_q[T]$ is irreducible in $\FF_q[T]$ but reducible in $\FF_q[S]$, then $\Legendre{P}{T}=1$. 
\end{lemma}

\begin{proof}
Since $P$ is reducible in $\FF_q[S]$, 
Lemma~\ref{lem:inertinquadratcextension} gives $b_q(P)=1$, so we may write 
\[
 P=A^2+ T B^2,
\]
$A,B\in \FF_q[T]$. Since $P\neq T$, we get that $T\nmid A$, so $P$ is a nonzero quadratic residue modulo $T$. 
\end{proof}

\begin{lemma}\label{lem:Reducible}
Let $P\in \FF_q[T]$. If $\Legendre{P}{T}=1$, then $P$ is reducible in $\FF_q[S]$.
\end{lemma}

\begin{proof}
Assume that $P$ is irreducible in $\FF_q[S]$. 
By \eqref{LegendrePT} and by assumption it follows that 
\[
\Legendre{- T}{P}=\Legendre{P}{T} =1,
\] 
so there exists $A\in \FF_q[T]$ such that $-T\equiv A^2\mod P$; i.e.,
\[
P\mid A^2 + T = (A+S)(A-S).
\]
Since $P$ is irreducible in $\FF_q[S]$, it must divide one of the factors, say $P\mid A+S$. Applying conjugation gives that $P\mid A-S$ as well, so $P\mid 2S$. This is a contradiction since $\deg_S P=2\deg_TP\geq 2$. So $P$ is reducible in $\FF_q[S]$.
\end{proof}

\begin{proposition}\label{prop:primesinFqS}
Let $h(S) \in \FF_q[S]$ be a prime polynomial in $S$. Then exactly one of the following hold:
\begin{enumerate}
\item $N(h)=\pm P$, where $P$ is a prime of $\FF_q[T]$ with $\Legendre{P}{T}=1$. 
\item $h=\pm Q\in \FF_q[T]$ where $Q$ is a prime in $\FF_q[T]$ with $\Legendre{Q}{T}=-1$. In particular $N(h)=\pm Q^2$.
\item $N(h) = - T$.
\end{enumerate}
\end{proposition}

\begin{proof}
Put $h=A+SB$ and $N(h)= (A+SB) (A-SB)=A^2+ TB^2 $. We apply freely the previous lemmas. 

Assume a prime $Q\in \FF_q[T]$ with $\Legendre{Q}{T}=-1$ divides $N(h)$. Then by Lemma~\ref{lem:irreducibility}, $Q$ is irreducible also in $\FF_q[S]$. Therefore, $Q$ divides one of the factors $A\pm SB$,  and by applying conjugation, $Q$ divides also the other. In particular,  $Q\mid h$. As $h$ is prime in $\FF_q[S]$, we have $h=c Q$, $c\in \FF_q$. The exact value of $c$ is determined by comparing the leading coefficients, and we are in case (2). 

Assume a prime $P\in \FF_q[T]$ with $\Legendre{P}{T} = 1$ divides $N(h)$. Then, by Lemma~\ref{lem:inertinquadratcextension}, $cP=(C+SD)(C-SD)$ for some nonzero $c\in \FF_q$, $C,D\in \FF_q[T]$, where $C+SD$ is irreducible. Thus, $C+SD$ divides either $A+SB$ or $A-SB$, so $C+SD = c'(A\pm SB)$, for some $c'\in \FF_q$ and we are in case (1).

Finally if we assume that $T$ divides $N(h)$, then $S$ must divide $A+SB$ or $A-SB$, and the same reasoning as above gives that $S=c(A+SB)$, for some $c\in \FF_q$, so $A=0$ and we are in case (3). 
\end{proof}

Now we can give a multiplicative description of $b_q$.

\begin{theorem}\label{thm:Fermat}
Let $f\in \EM_{n,q}$. Then the following are equivalent.
\begin{enumerate}
\item $b_q(f)=1$ 
\item Every prime $Q$ with $\Legendre{Q}{T}=-1$ appears with even multiplicity in the prime factorization of $f$.
\item Every prime $Q$ with $Q(-S^2)$ irreducible in $S$ appears with even multiplicity in the prime factorization of $f$. 
\end{enumerate}
\end{theorem}

\begin{proof}
(1)$\Rightarrow$(2): Assume that $b_q( f)=1$. Then  $f =N(h)$ for some $h\in \FF_q[S]$. By Proposition~\ref{prop:primesinFqS} we can write the prime factorization of  $h$ as 
\[
h=c S^a (A_1+SB_1)^{b_{11}}(A_1-SB_1)^{b_{12}} \cdots (A_r+SB_r)^{b_{r1}}(A_r-SB_r)^{b_{r2}} Q_1^{c_1} \cdots Q_k^{c_k}, \qquad a,b_{ij},c_{l}\geq 0
\]
where $N(A_i\pm SB_i)=\pm P_i$, for a prime $P_i$ of $\FF_q[T]$ with $\Legendre{P_i}{T}=1$ and $Q_i\in \FF_q[T]$ is a prime of $\FF_q[T]$ with $\Legendre{Q_i}{T}=-1$. Taking norms we get that
\[
f = c' T^{a} P_1^{b_{11}+b_{12}}\cdots P_r^{b_{r1}+b_{r2}} Q_1^{2c_1}\cdots Q_k^{2c_k}.
\]
So the multiplicities of the prime factors with $\Legendre{Q}{T}=-1$ are even.

(2)$\Rightarrow$(1):
If 
\[
f(T)=T^a P_1^{b_1} \cdots P_r^{b_r} Q_1^{2c_1}\cdots Q_k^{2c_k}
\] 
is the prime factorization of $f$, where $\Legendre{P_i}{T}=1$ and $\Legendre{Q_j}{T}=-1$, then $f$ is a norm since $T$, each of the $P_i$, and each of the $Q_j^2$ is a norm (by Proposition~\ref{prop:primesinFqS}).

(2)$\Leftrightarrow$(3): By Lemmas~\ref{lem:irreducibility} and \ref{lem:Reducible}, $\Legendre{Q}{T}=-1$ if and only if $Q$ remains irreducible in $\FF_q[S]$. Since $T=-S^2$, the latter is the same as saying that $Q(-S^2)$ is irreducible as a polynomial in $S$. 
\end{proof}

\section{Large Finite Field Limit}

Throughout this section the letter $P$ is reserved for primes with $\Legendre{P}{T}=(-1)^{\frac{q-1}{2}}$ and the letter $Q$ for primes with $\Legendre{Q}{T}=(-1)^{\frac{q-3}{2}}$.

Let $\F_n\subseteq \M_{n,q}$ be the subset of monic polynomials $f$ of degree $n$ such that every prime polynomial $Q$ divides $f$ with even multiplicity. 
So, by Theorem~\ref{thm:Fermat}, for $f\in \M_{n,q}$, we have $b_q( f)=1$ if and only if $f\in \F_n$. Thus 
\begin{equation}\label{FI0}
B_q(n) = \sum_{f\in \M_{n,q}} b_q(f) = \sum_{f\in \F_n}1=\#\F_n.
\end{equation}
Therefore to prove Theorem~\ref{thm:LF}, we have to estimate $\#\F_n$. 

\subsection{Partition of $\F_n$}
We break $\F_n$ into the following parts and evaluate each one separately:
\begin{align*} \label{eq:group_def}
 \F_{1,n}&= \{ f \in \F_n :  f =P_1 P_2\ldots P_r , P_i\neq P_j \ \forall i\neq j, r\geq 0\}, \\
 \F_{2,n}&= \{ f \in \F_n :  f =T P_1 P_2\ldots P_r,P_i\neq P_j \ \forall i\neq j, r\geq 0 \}, \\
 \F_{3,n}&= \{ f \in \F_n : f =P_1 P_2\ldots P_r Q_1^2, \deg{Q_1}=1, P_i\neq P_j \ \forall i\neq j,r\geq 0 \},  \\
 \F_{4,n}&= \{ f \in \F_n : f = P_1^2 P_2\ldots P_r,  \deg{P_1}=1,P_i\neq P_j\ \forall i\neq j ,r\geq 0\},\\
 \F_{5,n} &= \F_n \smallsetminus \bigcup_{i=1}^4 \F_{i,n}.
\end{align*}
It is clear that these sets are disjoint and that $\F_{n} = \bigcup_{i=1}^5 \F_{i,n}$, and so
\begin{equation}\label{FI}
B_q(n) = \#\F_n = \sum_{i=1}^5 \#\F_{i,n}.
\end{equation}
It is also clear that 
\begin{equation}\label{eq:F2F3}
\#\F_{2,n} = \#\F_{1,n-1} \qquad\mbox{and} \qquad \#\F_{3,n} = \frac{q-1}{2} \#\F_{1,n-2}=\frac{q}{2}\#\F_{1,n-2}+O(q^{n-2}).
\end{equation}

We start by bounding $\#\F_{5,n}$:
\begin{lemma}
\begin{equation}\label{eq:SD}
\# \F_{5,n} \leq \frac73 q^{n-2}.
\end{equation}
\end{lemma}

\begin{proof}
Clearly
\[
\#\{ f\in \F_{n} : T^2\mid f\} \leq \# \{ f\in \M_{n,q} : T^2\mid f\} = q^{n-2}.
\]
In a similar fashion, since there are $q^d$ options for monic $h$ of degree $d$, and for each $h$ there are at most $q^{n-2d}$ options for $f/h^2$, it follows that
\[
\begin{split}
\#\{ f\in \F_n : \exists h, \deg h>1, h^2\mid f\} &\leq \sum_{2\leq d\leq n/2}q^{d}\cdot q^{n-2d} \\&=q^n\sum_{2\leq d\leq n/2} q^{-d}\leq \frac{q^{n-2}}{1-q^{-2}} \leq \frac{4}{3}q^{n-2}.
\end{split}
\]

Now if $f\in \F_{5,n}$, then either $T^2$ divides $f$, or it has a divisor $h^2$, with $\deg(h)>1$. So we get the desired estimate. 
\end{proof}

\subsection{Estimating $\#\F_{1,n}$ and $\#\F_{4,n}$}
We use
the prime polynomial theorem in arithmetic progressions modulo $T$: 
Let $\alpha \in \{-1,1\}$; then 
\begin{equation}\label{PPTAP}
\pi_{q;\alpha}(n):=\#\left\{ R\in \M_{n,q} : \Legendre{R}{T} = \alpha, \mbox{ $R$ prime}\right\} = \frac{q^n}{2n} + O(q^{\lfloor n/2\rfloor}),
\end{equation}
as $q^n\to \infty$. 
We refer the reader to \cite[Theorem 4.8]{Rosen} for a proof dealing with an arithmetic progression with general modulus $m$ that gives an error term of $O(q^{n/2}/n)$. Careful examination of the proof with $m=T$, shows that in fact we get the error given in \eqref{PPTAP} since the $L$-function of the quadratic character modulo $T$ is $1$ (in the notation of the proof of \cite[Theorem 4.8]{Rosen}:  Indeed, for $\chi$ the non-trivial character modulo $T$, we have $C_{N}(\chi)=0$, $N\geq 1$ because the $C_N(\chi)$ are the $N$-th Taylor coefficients of the derivative of the the log of the $L$-function, which is zero since the $L$-function is $1$. So in equation 5 the error term is in fact $O(q^{\lfloor N/2\rfloor})$).

We need a bound on the error of \eqref{PPTAP} of size $O(q^{n-2})$ for which we have to take a more precise account when  $n=1,2$:
\begin{lemma}\label{lem:piqeps}
The following equalities hold. 
\begin{align}
\label{piqep1}\pi_{q;\alpha}(1)&= \frac{q-1}{2}\\
\label{piqep2} \pi_{q;\alpha}(2)&= \frac{1}{4}q^2-\frac{1+\alpha}{4} q +\frac{\alpha}{4}\\
\label{piqepj}\pi_{q;\alpha}(n)&= \frac{q^n}{2n}+O(q^{n-2}), & n\geq 3.
\end{align}
\end{lemma}

\begin{proof}
Since for $a\in \FF_q$ we have $\Legendre{T-a}{T}=\Legendre{-a}{T}$ and $T-a$ is always prime, we get that 
\[
\pi_{q;\alpha}(1) = \#\{a\in \FF_q : \Legendre{a}{T}=\alpha\} = \frac{q-1}{2}. 
\]

Next we consider quadratic polynomials. Let $\chi$ be the quadratic multiplicative character on $\FF_q$, i.e.\ $\chi(u)=1$ if $u$ is a nonzero square and $\chi(u)=-1$ if $u$ is not a square. 
There are $q\cdot \frac{q-1}{2}$ polynomials $f=T^2+aT+b$ with $\Legendre{f}{T}=\chi(b)=\alpha$. 
The number $M$ of ordered pairs $(u,v)$ for which $\chi(uv)=\alpha$ satisfies 
\[
2 M=\sum_{u\in \FF_q^*}\sum_{v\in \FF_q^*}(1 + \alpha \chi(uv)) = (q-1)^2  
\]
due to the orthogonality relations of characters and since $1+\alpha \chi(w)=2$ if $\chi(w)=\alpha$ and $1+\alpha\chi\left(w\right)=0$ otherwise.

Since $f$ is reducible if and only if $f=(T-u)(T-v)$ and since $M$ counts each such reducible once if $u=v$ and twice if $u\neq v$, we get that 
\[
\pi_{q;\alpha}(2) = q\cdot \frac{q-1}{2} - M/2 -L/2,
\]
where $L$ is the number of polynomials $f$ of the form $f=(T-u)^2$, i.e.\ the number of $u$ with $\chi(u^{2})=\alpha$. Clearly $L=q-1$ if $\alpha=1$ and $L=0$ if $\alpha =-1$. Putting all of this together yields
\begin{align*}
\pi_{q;\alpha}(2)&= q\cdot \frac{q-1}{2} - \frac{(q-1)^2}{4} - \frac{1+\alpha}{4}(q-1)\\
&=\frac{q-1}{2}\cdot \frac{q+1}{2}-\frac{1+\alpha }{4}(q-1)\\
&= \frac{1}{4}q^2-\frac{1+\alpha}{4} q +\frac{\alpha}{4}.
\end{align*}

The case $n\geq 3$ follows from \eqref{PPTAP} since $\lfloor \frac{n}{2}\rfloor\leq n-2$.
\end{proof}

A \emph{partition} $\lambda\vdash n$ of a positive integer $n$ is an $n$-tuple $\lambda = (\lambda_1, \ldots, \lambda_n)$ of non-negative integers such that $\sum_{j=1}^n j\lambda_j =n$, that is, $\lambda_j$ counts the number of parts of cardinality $j$ in a partition of the set $\{1,\ldots, n\}$.
For example $\left(n,0,...,0\right)\vdash n$ and also $\left(0,...,0,1\right)\vdash n$.  
We define 
\begin{equation}\label{defhn}
h_n = 
\sum_{\lambda\vdash n} \prod_{j=1}^n \frac{1}{\lambda_j! (2j)^{\lambda_j}}.
\end{equation}

\begin{lemma}
We have
\begin{equation}\label{eq:F1}
\#\F_{1,n} = h_n q^n - \left(\frac{1}{2}h_{n-1}+\frac{3}{4}h_{n-2}\right)q^{n-1}+O_n(q^{n-2}),
\end{equation}
where $h_{n}$ is the constant from \eqref{defhn}.
\end{lemma}

\begin{proof}
If $f=P_1\cdots P_r$, with distinct $P_i$ satisfying $\Legendre{P_i}{T}=1$, then the degrees of the $P_i$'s define a partition of $\lambda \vdash n$ by $\lambda_j=\#\{i:\deg P_i=j\}$. Counting the elements in $\F_{1,n}$ according to the partitions gives
\begin{equation}\label{eqF1-1}
\#\F_{1,n}  = 
\sum_{\lambda\vdash n} \prod_{j=1}^n \binom {\pi_{q;1}(j)} {\lambda_j} ,
\end{equation}
where $\pi_{q;1}(j)$ is the number of primes $P$ of degree $j$ with $\Legendre{P}{T}=1$. 

Using the formulas 
\[
\binom {x}{\lambda} = \frac{x^{\lambda}}{{\lambda}!} - \frac{\lambda(\lambda-1)}{2} \frac{x^{\lambda-1}}{\lambda!}+O(x^{\lambda-2}) \qquad \mbox{and}\qquad (x+y)^{\lambda}=x^{\lambda} + \lambda x^{\lambda-1} y + O(x^{\lambda-2})
\] 
for $\lambda\geq 0$ and $x\to \infty$, by \eqref{piqep1} we have 
\begin{equation}\label{eq:eq1}
\begin{split}
\binom {\pi_{q;1}(1)} {\lambda_1} 
		&= \frac{(q-1)^{\lambda_1}}{2^{\lambda_1}\lambda_1!}-\frac{\lambda_1(\lambda_1-1)(q-1)^{\lambda_1-1}}{2^{\lambda_1}\lambda_1!} +O(q^{\lambda_1-2})  \\
		&= \frac{1}{2^{\lambda_1}\lambda_1!}q^{\lambda_1} - \frac{\lambda_1+\lambda_1(\lambda_1-1)}{2^{\lambda_1}\lambda_1!}q^{\lambda_1-1}+O(q^{\lambda_1-2}).
\end{split}
\end{equation}

Similarly, by \eqref{piqep2} we have
\begin{equation}\label{eq:eq2}
\binom{\pi_{q;1}(2)}{\lambda_2} = \frac{1}{(2\cdot2)^{\lambda_2}\lambda_2!}q^{2\lambda_2} - \frac{2\lambda_2}{(2\cdot2)^{\lambda_2}\lambda_2!}q^{2\lambda_2-1} + O(q^{2\lambda_2 -2})
\end{equation}
and by \eqref{piqepj}, for $j\geq 3$ we have 
\begin{equation}\label{eq:eq3}
\binom {\pi_{q;1}(j)} {\lambda_j} = \frac{1}{(2j)^{\lambda_j} \lambda_j!}\cdot q^{j\lambda_j} + O(q^{j\lambda_j-2}).
\end{equation}
Applying \eqref{eq:eq1}, \eqref{eq:eq2}, and \eqref{eq:eq3} (and recalling that $n=\sum j\lambda_j$) we get, for a partition $\lambda$,
\begin{equation}
\begin{split}
\prod_{j=1}^n \binom {\pi_{q;1}(j)} {\lambda_j} 
		&= \left(\frac{1}{2^{\lambda_1}\lambda_1!}q^{\lambda_1} - \frac{\lambda_1+\lambda_1(\lambda_1-1)}{2^{\lambda_1}\lambda_1!}q^{\lambda_1-1}\right) \\
		&\quad\times \left(\frac{1}{(2\cdot2)^{\lambda_2}\lambda_2!}q^{2\lambda_2} - \frac{2\lambda_2}{(2\cdot2)^{\lambda_2}\lambda_2!}q^{2\lambda_2-1}\right) \\
		&\quad \times \prod_{j\geq 3} \frac{1}{(2j)^{\lambda_j} \lambda_j!}\cdot q^{j\lambda_j} + O(q^{n-2})\\
		&=q^n \prod_{j=1}^n \frac{1}{(2j)^{\lambda_j} \lambda_j!} -(2A_{\lambda}+B_{\lambda})q^{n-1} +O(q^{n-2}),\\
\end{split}
\end{equation}
where 
\begin{align*}
A_{\lambda}&=	\frac{\lambda_2}{(2\cdot2)^{\lambda_2}\lambda_2!}\prod_{j\neq 2}\frac{1}{(2j)^{\lambda_j} \lambda_j!} = \lambda_2 \prod_{j=1}^n \frac{1}{(2j)^{\lambda_j} \lambda_j!},\\
B_{\lambda}&= \frac{\lambda_1+\lambda_1(\lambda_1-1)}{2^{\lambda_1}\lambda_1!}
\prod_{j>1}\frac{1}{(2j)^{\lambda_j} \lambda_j!} =\lambda_1^2 \prod_{j=1}^n \frac{1}{(2j)^{\lambda_j} \lambda_j!} . 
\end{align*}
Since by definition $h_n = \sum_{\lambda\vdash n} \prod_{j=1}^n \frac{1}{(2j)^{\lambda_j} \lambda_j!}$, we have
\begin{equation}\label{eqF1n}
\#\F_{1,n} =h_n q^{n}  -  \left(2\sum_{\lambda\vdash n} A_{\lambda} + \sum_{\lambda\vdash n} B_{\lambda}\right) q^{n-1} +O(q^{n-2})
\end{equation} 
We note that $A_{\lambda}=0$ if $\lambda_2=0$ and  if $\lambda_2>0$ then $\lambda_n=\lambda_{n-1}=0$ and so 
\[
A_{\lambda} = \frac{1}{4}\prod_{j=1}^{n-2} \frac{1}{(2j)^{\mu_j} \mu_j!},
\]
where $\mu\vdash(n-2)$ is the partition defined by $\mu_{j}=\lambda_j$ for $j\neq 2$ and $\mu_2=\lambda_2-1$.
Thus,
\begin{equation}\label{eq:Alambda}
\sum_{\lambda\vdash n} A_\lambda = \sum_{\lambda\vdash n, \lambda_2>0} A_{\lambda} = \frac{1}{4}\sum_{\mu\vdash(n-2)}\prod_{j=1}^{n-2} \frac{1}{(2j)^{\mu_j} \mu_j!} = \frac14 h_{n-2}.
\end{equation}
For the $B_{\lambda}$'s we proceed in a similar manner. If $\lambda_1=0$, then $B_{\lambda}=0$; if $\lambda_1=1$ then $\lambda_n=0$ and so
\[
B_{\lambda} =\frac12 \prod_{j=1}^{n-1} \frac{1}{(2j)^{\kappa_j} \kappa_j!};
\]
and if $\lambda_1\geq 2$, then 
\[
B_{\lambda} = \frac{1}{2} \prod_{j=1}^{n-1}\frac{1}{(2j)^{\kappa_j} \kappa_j!} + \frac{1}{4}\prod_{j=1}^{n-2}  \frac{1}{(2j)^{\nu_j} \nu_j!},
\]
where $\kappa\vdash (n-1)$ and $\nu\vdash(n-2)$ are defined by $\kappa_1+1=\nu_1+2=\lambda_1$ and $\kappa_j=\nu_j=\lambda_j$, $j\geq 2$. We see that as $\lambda$ varies over all partitions of $n$, $\kappa$ and $\nu$ vary over all partitions of $n-1$ and $n-2$, respectively.  Thus,
\begin{equation}\label{eq:Blambda}
\sum_{\lambda\vdash n}  B_\lambda = \frac{1}{2}\sum_{\kappa\vdash (n-1)}\prod_{j=1}^{n-1}\frac{1}{(2j)^{\kappa_j} \kappa_j!} + \frac{1}{4} \sum_{\nu\vdash(n-2)}\prod_{j=1}^{n-2}  \frac{1}{(2j)^{\nu_j} \nu_j!} = \frac12 h_{n-1} + \frac{1}{4}h_{n-2}.
\end{equation}
Plugging in \eqref{eq:Alambda} and \eqref{eq:Blambda} in \eqref{eqF1n} gives the assertion.
\end{proof}

\begin{lemma} \label{lemma:F_4_n}
\begin{equation}\label{eq:F4}
\#F_{4,n}=\frac{1}{2}h_{n-2}q^{n-1} +O_n(q^{n-2})
\end{equation}
\end{lemma}

\begin{proof}
Assume $f\in \F_{4,n}$, i.e.\ $f=P_1^{2} P_2\cdots P_r$ with $\deg P_1=1$ and $P_i\neq P_j$ for $i\neq j$. Thus the degrees of $P_2,\ldots, P_r$ determine a partition $\lambda=(\lambda_1,\ldots, \lambda_{n-2})\vdash n-2$ where $\lambda_1$ is the number of linear polynomials other than $P_1$ appearing in the factorization of $f$. Thus, since there are $\left(\frac{q-1}{2}\right)$ linear polynomials $P_1$, we get
\[
\#F_{4,n}
		 = \left(\frac{q-1}{2}\right)\sum_{\lambda \vdash {n-2}} \binom {\pi_{q;\epsilon}(1)-1} {\lambda_1} \prod_{j=2}^{n-2}\binom {\pi_{q;\epsilon}(j)} {\lambda_j}.
\]
Using the formula $\binom{x}{\alpha}=x^{\alpha}+O(x^{\alpha-1})$ and the estimates in Lemma~\ref{lem:piqeps} we have
\[
\#F_{4,n} = \frac{q}{2}\left(\sum_{\lambda \vdash {n-2}}  \prod_{j=1}^{n-2}\frac {\left(\frac{q^j}{2j}\right)^{\lambda_j}} {\lambda_j!}\right)+O(q^{n-2})  =\frac{1}{2}h_{n-2}q^{n-1} +O(q^{n-2}),
\]
as needed.
\end{proof}

\subsection{Proof of Theorem~\ref{thm:LF}}
By \eqref{FI}, \eqref{eq:F2F3}, \eqref{eq:SD},  \eqref{eq:F1}, and \eqref{eq:F4} we have 
\[
\begin{split}
B_q(n) &= h_n q^n + \left(\frac{1}{2}h_{n-1} -\frac{1}{4}h_{n-2} \right)q^{n-1}+ \frac{1}{2}h_{n-2}q^{n-1} +O(q^{n-2})\\
&= h_nq^n + 
\left(\frac{1}{2}h_{n-1} + \frac{1}{4}h_{n-2}\right)q^{n-1} + O(q^{n-2}).
\end{split}
\]
To conclude the proof we note that $h_n$ is the normalization factor in Ewens' sampling formula, hence it satisfies $h_n = \frac{1}{4^n}\binom{2n}{n}$, see \cite[Eq.~3]{KM} for a proof of the last equation. 
\qed

\section{
	Large Degree
}
\label{sec:LD_pf}
Our proof of Theorem~\ref{thm:xLandauLD} follows the proof of Landau's classical
theorem regarding integer sums of two squares, see \cite{Landau}, \cite[pages 257-263]{LeVeque}, or the detailed instructed exercise \cite[exercise 21 on page 187]{MV}. First we define a generating
function for our counting function $B_{q}(n)$. After a change of
variables we use the arithmetic properties of $b_{q}(f)$ to show
that our function is in fact composed of the zeta function for $\mathbb{F}_{q}\left[T\right]$,
and the Dirichlet $L$-series associated with the quadratic character
modulo $T$. We use this representation to show analytic properties
of our generating function, and using Cauchy's integration formula
we translate our counting problem to a calculation of integrals along
carefully defined curves in the complex plane. Approximating these
integrals will give the desired estimate and prove the theorem.

%
%

\subsection{Dirichlet Series}
Consider the generating function of $B_q(n)$:
\[
D^*(u) = \sum_{n=0}^\infty B_q(n)u^n.
\]
Since $B_q(n)\leq q^n$, the series converges  absolutely and uniformly on compact subsets of  the open disc $\{u\in \mathbb{C}:|u|<q^{-1}\}$ and we have in this region
\begin{equation}\label{eq:D*bound}
|D^*(u)| \leq \sum_{n=0}^\infty q^n |u|^n =\frac{1}{1-q|u|}
\end{equation}
Putting $u=q^{-s}$, for $s$ with $\Re(s) >1$ we have
\[
D(s)=\sum_{f \mbox{ \tiny monic}}b_q(f)\left|f\right|^{-s} = \sum_{n=0}^{\infty} B_q(n) q^{-ns} = D^*(u).
\]
Since $b_q$ is  multiplicative,  we have the factorization 
\[
D(s)=\prod_{P}\left(\sum_{k=0}^{\infty}{b(P^{k})}{\left|P\right|^{-ks}}\right),
\]
where the product is over all monic prime polynomials.  Theorem~\ref{thm:Fermat}
gives that 
\[
b(P^{k})=\begin{cases}
1  ,&\Legendre{P}{T}=1\\
1  ,& P=T\\
1  ,&\Legendre{P}{T}=-1\,\mbox{and \ensuremath{2\mid k}}\\
0  ,&\Legendre{P}{T}=-1\,\mbox{and \ensuremath{2\nmid k}}
\end{cases}
\]
and so 
\begin{align*}
D
(s) 
		& =  \left(\sum_{k=0}^{\infty}\left|T\right|^{-ks}\right)\prod_{\Legendre{P}{T}=1}\left(\sum_{k=0}^{\infty}\left|P\right|^{-ks}\right)\prod_{\Legendre{Q}{T}=-1}\left(\sum_{k=0}^{\infty}\left|Q\right|^{-2ks}\right)\\
	 	& =  \left(1-q^{-s}\right)^{-1}\prod_{\Legendre{P}{T}=1}\left(1-\left|P\right|^{-s}\right)^{-1}\prod_{\Legendre{Q}{T}=-1}\left(1-\left|Q\right|^{-2s}\right)^{-1}.
\end{align*}

We will represent $D
$ using Dirichlet series which are easier
to manipulate. We recall the basic properties of the zeta function and of Dirichlet $L$-series; we refer the reader to  \cite{Rosen} for a detailed exposition that contains proofs.
First recall  the definition of the zeta function for $\FF_q[T]$:
\[
\zeta_q(s)=\sum_{f \mbox{ \tiny monic}} |f|^{-s} = \prod_{P}\left(1-\left|P\right|^{-s}\right)^{-1},
\]
where 
$P$ runs over monic primes. It is easy to see that 
\[
\zeta_{q}(s) = \frac{1}{1-q^{1-s}}.
\]

Let $\chi\colon \FF_q[T] \to \mathbb{C}$ be the quadratic Dirichlet character modulo $T$ (i.e.\ $\chi(f) \equiv \Legendre{f}{T}\mod p$, and $p={\rm char} (\FF_q)$). Denote by $L(s)$ the associated Dirichlet $L$-series:
\[
L(s)=\sum_{f \mbox{ \tiny monic}}\chi(f)\left|f\right|^{-s}=\prod_{\Legendre{P}{T}=1}\left(1-\left|P\right|^{-s}\right)^{-1}\prod_{\Legendre{Q}{T}=-1}\left(1+\left|Q\right|^{-s}\right)^{-1}.
\]
Since $\chi$ is defined modulo a linear polynomial, we have 
\[
L(s)=1.
\]
Precisely as done by Landau in the number field setting, we arrive at the identity 
\begin{equation}\label{eq:D2}
D
^{2}(s)=\zeta_q(s)L(s)\varphi(s)=\frac{\varphi(s)}{1-q^{1-s}},
\end{equation}
for $\Re(s)>1$, 
where 
\[
\varphi(s)=\left(1-q^{-s}\right)^{-1}\prod_{\Legendre{Q}{T}=-1}\left(1-\left|Q\right|^{-2s}\right)^{-1}.
\]
We note that the expression given in \eqref{eq:D2} is simpler than the corresponding one in number fields, due to the relative simplicity of $\zeta_q$ and $L$. 
The product defining $\varphi$ converges absolutely and uniformly on compact subsets of  $\{s:\Re(2s)>1\}$, so $\varphi$ is a non-vanishing analytic function in the region $\Re(s)>1/2$.  
Moreover, if $\Re(s)\geq 1-\delta>1/2$, then 
\begin{equation}\label{bound_varphi}
\begin{split}
|\varphi(s)| &= |1-q^{-s}|^{-1} \prod_{\Legendre{Q}{T}=-1}|1-|Q|^{-2s}|^{-1} \leq \frac{|\zeta(2s)|}{1-q^{\delta-1}}  \\
&\leq \frac{1}{(1-q^{\delta-1})(1-q^{2\delta-1})}\leq \frac{1}{(1-q^{2\delta-1})}.
\end{split}
\end{equation}
Since 
\[
\psi(s)=\begin{cases}
\frac{s-1}{1-q^{1-s}},& s\neq 1
\\
\frac{1}{\log q}, &s=1
\end{cases}
\]
is a non-vanishing analytic function in the region 
\[
\B=\left \{s : \Re(s)>1/2 \mbox{ and }  |\Im(s)|< \frac{\pi i}{\log q} \right\} ,
\]
we get that 
$\varphi(s)\psi(s)$ has an analytic square root in $\B$; so we may write
\[
\sqrt{\varphi(s)\psi(s)} = a_0 + (s-1) E(s),
\] 
where $E(s)$ is analytic in $\B$. Note that 
\[
a_0 = \sqrt{\varphi(1)\psi(1)} = \sqrt{\frac{\left(1-q^{-1}\right)^{-1}\prod_{\Legendre{Q}{T}=-1}\left(1-\left|Q\right|^{-2}\right)^{-1}}{\log q}}= \frac{K_q}{\sqrt{\log q}}.
\]
Thus, by \eqref{eq:D2}, for $s\in \B \smallsetminus (-\infty,1]$, we have 
\begin{equation}\label{eq:Dapprox}
D(s) = \sqrt{\frac{\varphi(s)\psi(s)}{s-1}} = \frac{K_q}{\sqrt{\log q}}\frac{1}{\sqrt{s-1}}+ E(s)\sqrt{s-1}.
\end{equation}

\subsection{Contours}
We fix a small parameter $0<\delta <1/2$, and we take a parameter $\varepsilon>0$ that will tend to zero.  
Consider the contour defined by the following curves as shown in Figure~\ref{fig:1}:
\begin{enumerate}
\item[]  $\gamma_{1}$ = the straight line from $\left(1-\delta\right)+\frac{\pi}{\log(q)}i$
to $\left(1-\delta\right)+\varepsilon i$.
\item[]  $\gamma_{2}$ = the straight line from $\left(1-\delta\right)+\varepsilon i$
to $1+\varepsilon i$.
\item[]  $\gamma_{3}$ = the semi circle of radius $\varepsilon$ from
$1+\varepsilon i$ to $1-\varepsilon i$.
\item[]  $\overline{\gamma}_{2}$ = the straight line from $1-\varepsilon i$
to $\left(1-\delta\right)-\varepsilon i$.
\item[]  $\overline{\gamma}_{1}$ = the straight line from $\left(1-\delta\right)-\varepsilon i$
to $\left(1-\delta\right)-\frac{\pi}{\log(q)}i$.
\end{enumerate}
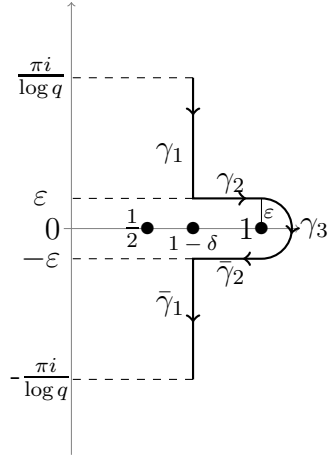
\begin{figure}[!h] 
\begin{tikzpicture}[decoration={markings,
mark=at position 0.5cm with {\arrow[line width=1pt]{>}},
mark=at position 2.3cm with {\arrow[line width=1pt]{>}},
mark=at position 3.2cm with {\arrow[line width=1pt]{>}},
mark=at position 4cm with {\arrow[line width=1pt]{>}},
mark=at position 5.5cm with {\arrow[line width=1pt]{>}}
}
]
\draw[help lines,->] (-0.1,0) -- (3,0) coordinate (xaxis);
\draw[help lines,->] (0,-3) -- (0,3) coordinate (yaxis);

\path[draw, line width=0.8pt,postaction=decorate] 
(1.6,2) node[left] {} 
-- (1.6,0.4) node[left] {} 
-- (2.5,0.4) node[above] {} arc (90:-90:0.4) 
-- (1.6,-0.4) node[right] {} 
-- (1.6,-2) node[below] {}
;

\path[draw, dashed] (1.6,0.4) node[below] {} -- (0,0.4);
\node at (-.4,.4) {$\varepsilon$} ;
\path[draw, dashed] (1.6,2) node[below] {} -- (0,2);
\node at (-.4,2) {$\frac{\pi i}{\log q}$} ;
\path[draw, dashed] (1.6,-0.4) node[below] {} -- (0,-0.4);
\node at (-.4,-.4) {$-\varepsilon$} ;
\path[draw, dashed] (1.6,-2) node[below] {} -- (0,-2);
\node at (-.4,-2) {-$\frac{\pi i}{\log q}$} ;

\path[draw, line width = 0.4pt] (2.5,0) node[below] {} -- (2.5,0.4);

\node at (2.5,0) {$\bullet$};
\node at (1,0) {$\bullet$};
\node at (1.6,0) {$\bullet$};

\node[left] {$0$};
\node at (2.3,0) {$1$};
\node at (0.8,0) {$\frac{1}{2}$};
\node at (1.6,-.2) {{\tiny $1-\delta$}};
\node at (1.3,1) {$\gamma_1$};
\node at (1.3,-1) {$\bar{\gamma}_1$};
\node at (2.1,0.6) {${\gamma}_2$};
\node at (2.1,-0.6) {$\bar{\gamma}_2$};
\node at (3.2,0) {$\gamma_3$};
\node at (2.6,0.2) {{\tiny $\varepsilon$}};

\end{tikzpicture}

\caption{The contour in the $s$-plane}
\label{fig:1}
\end{figure}

Under the transformation $s\to u=q^{-s}$ these curves turn
into the following curves, as shown in Figure~\ref{fig:2}:
\begin{enumerate}
\item[]  $\Gamma_{1}$ = semi circle  of radius $q^{\delta-1}$ from $-q^{\delta-1}$
to $q^{\delta-1}e^{-i\varepsilon\log q}$.
\item[] $\Gamma_2$ = straight line  from $q^{\delta-1}e^{-i\varepsilon\log q}$
to $q^{-1}e^{-i\varepsilon\log q}$.
\item[] $\Gamma_{3}$ = arc from $q^{-1}e^{-i\varepsilon\log q}$ to $q^{-1}e^{i\varepsilon\log q}$
through $q^{-1-\varepsilon}$.
\item[] $\overline{\Gamma}_{2}$ = straight line  from $q^{-1}e^{i\varepsilon\log q}$
to $q^{\delta-1}e^{i\varepsilon\log q}$.
\item[] ${\overline{\Gamma}}_{1}$ = semi circle of radius $q^{\delta-1}$
from $q^{\delta-1}e^{i\varepsilon\log q}$ to $-q^{\delta-1}$.
\end{enumerate}

\begin{figure}[t]

\begin{tikzpicture}[decoration={markings,
mark=at position 3cm with {\arrow[line width=1pt]{<}},
mark=at position 6.2cm with {\arrow[line width=1pt]{<}},
mark=at position 6.8cm with {\arrow[line width=1pt]{<}},
mark=at position 7.5cm with {\arrow[line width=1pt]{<}},
mark=at position 10cm with {\arrow[line width=1pt]{<}}
}
]
\draw[help lines,->] (-3,0) -- (3,0) coordinate (xaxis);
\draw[help lines,->] (0,-3) -- (0,3) coordinate (yaxis);

\path[draw, line width=0.8pt,postaction=decorate] 
(-2,0) node[above] {} arc (180:10:2) -- (10:1.5)  .. controls (1.2,0) .. (-10:1.5) -- 
(-10:2) arc (-10:-180:2)
;

\node at (1.5,0) {$\bullet$};

\draw[line width=0.2pt] (0,0) circle (1.5);
\draw[line width=0.2pt] (0,0) circle (1.25);
\draw[line width=0.2pt] (0,0) -- (120:2) node[right] {$q^{\delta-1}$};
\draw[line width=0.2pt] (0,0) -- (140:1.5) node[right] {$q^{-1}$};
\draw[line width=0.2pt] (0,0) -- (200:1.25) node[right] {$q^{-1-\varepsilon}$};

\node at (0,2.4) {$\bar{\Gamma}_1$};
\node at (0,-2.4) {$\Gamma_1$};
\node at (1.7,0.6)  {$\bar{\Gamma}_2$};
\node at (1.6,-0.6) {${\Gamma}_2$};
\node at (1,0) {$\Gamma_3$};

\end{tikzpicture}
\caption{The contour $C_{\delta}$ in the $u$-plane}\label{fig:2}
\end{figure}
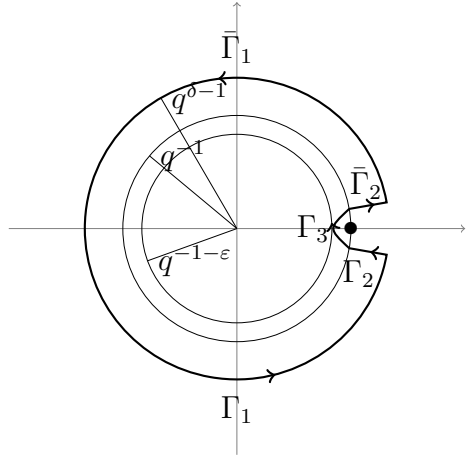
Denote by $C_{\delta}$ the contour defined by these curves directed
counterclockwise. 
Since $D(s)$ is analytic in $\B\smallsetminus (-\infty,1]$, we get that 
$D^*(u)$ is analytic inside $C_{\delta}$.   
By Cauchy's integration formula we have 
\begin{equation}\label{eq:CF}
B_q(n) = \frac{1}{2\pi i}\int\limits _{C_{\delta}}\frac{D^*(u)}{u^{n+1}}du.
\end{equation}
We note that Cauchy's integration formula comes
in place of Perron's formula in the proof of Landau's classical theorem.

Next we turn to calculate the integral along each of the curves composing $C_{\delta}$.
As in the classical
proof, the integrals which play the major role in our calculation
of the main term and error term stated in Theorem~\ref{thm:xLandauLD} are the integrals
along the curves $\Gamma_{2}$ and $\overline{\Gamma}_{2}$ (which
are the images of the curves $\gamma_{2}$ and $\overline{\gamma}_{2}$
under the transformation $s\rightarrow q^{-s})$, where the integrals
along the rest of $C_{\delta}$ will be roughly approximated and their
values absorbed in the error term.

\subsubsection{The integrals along $\Gamma_1$ and $\bar{\Gamma}_1$}
By \eqref{bound_varphi} and \eqref{eq:D2} we have 
\[
|D(s)|\leq M_{\delta,q}:= \frac{1}{\sqrt{q^\delta(1-q^{2\delta-1})}}
\]
on $\gamma_1 \cup \bar{\gamma}_1$.  Thus $|D^*(u)|\leq M_{\delta,q}$ on $\Gamma_1\cup \bar{\Gamma}_1$ and trivially one has $|1/u^{n+1}|=q^{(1-\delta)(n+1)}$ there; so 
\begin{equation}\label{int:Gamma1}
\left|\frac{1}{2\pi i} \int_{\Gamma_1\cup \bar{\Gamma}_1} \frac{D^*(u)}{u^{n+1} } du\right| \leq \frac{2\pi q^{\delta-1}}{2\pi} q^{(1-\delta)(n+1)} M_{\delta,q}\leq 
Cq^{(1-\delta)n},
\end{equation}
where $C>0$ depends on $\delta$ and $q$.

\subsubsection{The integral along $\Gamma_3$}
On $\gamma_3$ we have $|D(s)| \ll_{q} \frac{1}{\sqrt{\varepsilon}}$. Using  $|q^{ns}|\leq q^{n(1+\varepsilon)}$ and 
\begin{equation}\label{eq:changeofvariables}
du = -\log q \cdot q^{-s} 
ds,
\end{equation}
 we get that 
\begin{equation} \label{int:Gamma3}
\left|\frac{1}{2\pi i}\int_{\Gamma_3} \frac{D^*(u)}{u^{n+1}}du \right| =\frac{\log q}{2\pi } \left| \int_{\gamma_3} D(s)q^{ns} ds \right|\ll_{q,\delta}\frac{q^{n(1+\varepsilon)}\varepsilon}{\sqrt{\varepsilon}} =q^{n(1+\varepsilon)}\sqrt{\varepsilon}\to 0, 
\end{equation}
as $\varepsilon \to 0$.
\subsubsection{The integrals along $\Gamma_2$ and $\bar{\Gamma}_2$}

Let us abbreviate and write 
\begin{equation}\label{I_2}
I_{2,\varepsilon} = \frac{1}{2\pi i} \int_{\Gamma_2} \frac{D^*(u)}{u^{n+1}}du \quad \mbox{and} \quad \bar{I}_{2,\varepsilon} = \frac{1}{2\pi i} \int_{\bar{\Gamma}_2} \frac{D^*(u)}{u^{n+1}}du.
\end{equation}
By \eqref{eq:Dapprox} and \eqref{eq:changeofvariables} we have 
\begin{equation}\label{eq:I2epsplus}
\begin{split}
I_{2,\varepsilon} &= \frac{\log q}{2\pi i} \int_{1-\delta}^1 D(t+i\varepsilon)q^{n(t+i\varepsilon)} dt \\
&=\frac{\sqrt{\log q} K_q}{2\pi i} \int_{1-\delta}^1 \frac{q^{n(t+i\varepsilon)}}{\sqrt{t+i\varepsilon-1}} + \frac{\log q}{2\pi i} \int_{1-\delta} E(t+i\varepsilon)q^{n(t+i\varepsilon)}\sqrt{t+i\varepsilon-1}dt.
\end{split}
\end{equation}
We note that the two integrals are continuous at $\varepsilon = 0^{+}$, and so
\begin{equation}\label{eq:I2ep0}
\begin{split}
\lim_{\varepsilon\to 0^+}I_{2,\varepsilon} =\frac{\sqrt{\log q} K_q}{2\pi i} \int_{1-\delta}^1 \frac{q^{nt}}{\sqrt{t-1}} + \frac{\log q}{2\pi i } \int_{1-\delta}^1 E(t)q^{nt}\sqrt{t-1}dt.
\end{split}
\end{equation}
By Watson's lemma (see e.g.\ \cite[Page~103]{BH}), for $\alpha$, $x\geq 1$, $1>\delta>2\rho>0$, one has
\begin{equation}\label{eq:integration}
\int_{1-\delta}^1(1-t)^\alpha x^t dt = \int_{0}^\delta x^{1-u}u^\alpha du = \frac{x}{(\log x)^{\alpha+1}} \Gamma(\alpha+1) + O(x^{1-\rho}),
\end{equation}
as $x\to \infty$. Here $\Gamma(\alpha+1)=\int_{0}^\infty v^{\alpha} e^{-v}dv$ is the gamma function.

For the first integral, we see that by \eqref{eq:integration} and since $\Gamma(1/2)=\sqrt{\pi}$ we have 
\[
\begin{split}
\frac{\sqrt{\log q} K_q}{2\pi i } \int_{1-\delta}^1 \frac{q^{nt}}{\sqrt{t-1}}dt
&=\frac{1}{2}\frac{K_q}{\sqrt{\pi}}\frac{q^n}{n} + O(q^{n(1-\delta/4)}).
\end{split}
\]
For the second integral we use that $E$ is analytic, hence bounded in the region, say by $C=C_{\delta,q}$. So by \eqref{eq:integration} (with $\alpha=1/2$)
\[
\left|\int_{1-\delta}^1 E(t)q^{nt}\sqrt{t-1}dt\right|\leq C \frac{q^n}{n^{3/2}}.
\]
From this we derive
\begin{equation}\label{eq:calcal}
\lim_{\varepsilon\to 0^+}I_{2,\varepsilon} = \frac{1}{2}\frac{K_q}{\sqrt{\pi}}\frac{q^n}{n} + O_{\delta,q}\left(\frac{q^n}{n^{3/2}}\right), \qquad n\to \infty.
\end{equation}

The calculation of $\bar{I}_{2,\varepsilon}$ is similar to the above calculation of ${I}_{2,\varepsilon}$ with two changes: The first is that the direction of $\bar\gamma_2$ is opposite to that of $\gamma_2$ and the second is that instead of $\sqrt{t+i\varepsilon-1}$, we work with $\sqrt{t-i\varepsilon -1}$. Each of these changes contributes a change of sign in the limit $\varepsilon \to 0^+$, and so
\begin{equation}\label{eq:eqeqeq}
\lim_{\varepsilon\to 0^+} \bar{I}_{2,\varepsilon}=\lim_{\varepsilon\to 0^+} I_{2,\varepsilon}.
\end{equation}
By  \eqref{I_2},\eqref{eq:calcal} and \eqref{eq:eqeqeq} we have
\begin{equation}\label{eq:G2Gbar2}
\begin{split}
\lim_{\varepsilon\to 0^{+}} \frac{1}{2\pi i}\int_{\Gamma_2\cup \bar{\Gamma}_2} \frac{D^*(u)}{u^{n+1}} 
&= \frac{K_q}{\sqrt{\pi}}\frac{q^n}{n} + O_{\delta,q}\left(\frac{q^n}{n^{3/2}}\right), \qquad n\to \infty.
\end{split}
\end{equation}

\subsection{Proof of Theorem~\ref{thm:xLandauLD}}
By \eqref{eq:CF}, the definition of the contour $C_\delta = \Gamma_1\cup\Gamma_2\cup \Gamma_3\cup\bar\Gamma_2\cup \bar\Gamma_1$,   \eqref{eq:G2Gbar2}, \eqref{int:Gamma1}, and 
\[
\begin{split}
B_q(n) &= \frac{1}{2\pi i}\int\limits _{C_{\delta}}\frac{D^*(u)}{u^{n+1}}du \\
	&= \lim_{\varepsilon\to 0^+} 
	\frac{1}{2\pi i}\int\limits _{\Gamma_2\cup\bar{\Gamma}_2}\frac{D^*(u)}{u^{n+1}}du 
	+\frac{1}{2\pi i}\int\limits _{\Gamma_1\cup \bar{\Gamma}_1}\frac{D^*(u)}{u^{n+1}}du 
	+\frac{1}{2\pi i}\int\limits _{\Gamma_3}\frac{D^*(u)}{u^{n+1}}du \\
	&=
	\frac{K_q}{\sqrt{\pi}} \frac{q^n}{n} + O_{\delta,q}\left(\frac{q^n}{n^{3/2}}\right)  
	+ O(q^{(1-\delta)n}) + 0\\
	&= \frac{K_q}{\sqrt{\pi}} \frac{q^n}{n} + O_{\delta,q}\left(\frac{q^n}{n^{3/2}}\right) ,
\end{split}
\]
where all the implied constants depend on $q$ and $\delta$. This finishes the proof. \qed

\bibliographystyle{plain}
\bibliography{SoSbib}{}

\begin{thebibliography}{10}

\bibitem{BH}
Norman Bleistein and Richard~A. Handelsman.
\newblock {\em Asymptotic expansions of integrals}.
\newblock Dover Publications, Inc., New York, second edition, 1986.

\bibitem{Crane}
Harry Crane.
\newblock The ubiquitous ewens sampling formula.
\newblock {\em Statistical Science}, in print.

\bibitem{Ewens}
Warren~J. Ewens.
\newblock The sampling theory of selectively neutral alleles.
\newblock {\em Theoret. Population Biology}, 3:87--112, 1972.

\bibitem{HMW}
Jeffrey Hoffstein, Kathy~D. Merrill, and Lynne~H. Walling.
\newblock Automorphic forms and sums of squares over function fields.
\newblock {\em J. Number Theory}, 79(2):301--329, 1999.

\bibitem{KM}
Samuel Karlin and James McGregor.
\newblock Addendum to a paper of {W}. {E}wens.
\newblock {\em Theoret. Population Biology}, 3:113--116, 1972.

\bibitem{Landau}
Edmund Landau.
\newblock { \"{U}ber die Einteilung der positiven ganzen Zahlen in vier Klassen
  nach der Mindestzahl der zu ihrer additiven Zusammensetzung erforderlichen
  Quadrate}.
\newblock {\em Arch. Math. Phys.}, 13:305--312, 1908.

\bibitem{LeVeque}
William~J. LeVeque.
\newblock {\em Topics in number theory. {V}ol. {I}, {II}}.
\newblock Dover Publications, Inc., Mineola, NY, 2002.

\bibitem{MW}
Kathy~D. Merrill and Lynne~H. Walling.
\newblock Sums of squares over function fields.
\newblock {\em Duke Math. J.}, 71(3):665--684, 1993.

\bibitem{MV}
Hugh~L. Montgomery and Robert~C. Vaughan.
\newblock {\em Multiplicative number theory. {I}. {C}lassical theory},
  volume~97 of {\em Cambridge Studies in Advanced Mathematics}.
\newblock Cambridge University Press, Cambridge, 2007.

\bibitem{Rosen}
Michael Rosen.
\newblock {\em Number theory in function fields}, volume 210 of {\em Graduate
  Texts in Mathematics}.
\newblock Springer-Verlag, New York, 2002.

\end{thebibliography}

%
%
%
%
%
%
%
%

\end{document}